\documentclass[12pt]{amsart}
\usepackage{amscd}
\usepackage{verbatim}
\usepackage{amssymb, amsmath, amsthm, amscd,ifthen}
\usepackage[dvips]{graphics}
\usepackage[cp866]{inputenc}
\usepackage{graphicx}
\usepackage{epsfig}
\usepackage{xcolor}

\usepackage{amsmath,amssymb,amscd,}
\usepackage[all]{xy}

\usepackage{amssymb}

\newtheorem{thm}{Theorem}[section]

\newtheorem{prop}[thm]{Proposition}

\newtheorem{cor}[thm]{Corollary}

\theoremstyle{definition}

\newtheorem{dfn}[thm]{Definition}

\title{Envy-free division in the presence of a dragon}

\author[G. Panina]{Gaiane Panina}
\author[R. \v{Z}ivaljevi\'{c}]{Rade \v{Z}ivaljevi\'{c}}

\address[G. Panina]
{ St.\ Petersburg Department of Steklov Mathematical Institute}
\address[R. \v{Z}ivaljevi\'{c}]{Mathematical Institute of the Serbian Academy of Sciences and Arts (SASA), Belgrade}

\address{
  }

 \keywords{Envy-free division, Dragon marriage theorem, KKM theorem, configuration space/test map scheme}

\begin{document}

\date{December 24, 2021}

\begin{abstract}
Following a novel approach, where the emphasis is on configuration spaces and equivariant topology, we prove several results addressing the \emph{envy-free division problem} in the presence of an unpredictable (secretive, non-cooperative) player, called \emph{the dragon}. There are two basic scenarios.

 1. There are $r-1$ players and a dragon. Once the ``cake'' is divided into $r$ parts, the dragon makes his choice  and grabs one of the pieces. After that the players should be able to share the remaining pieces in an envy-free fashion.

  2. There are $r+1$ players who divide the cake into $r$ pieces. A ferocious dragon comes and swallows one of the players. The players need to cut the cake in advance in such a way that no matter who is the unlucky player swallowed by the dragon, the remaining players can share the tiles in an envy-free manner.

  We emphasize that  in both settings the players are allowed to choose degenerate pieces of the cake. Moreover, the players construct in advance both a cut of the cake and a \emph{decision tree}, allowing them to minimize the uncertainty of what pieces can be  given to each of the players.
\end{abstract}

 \maketitle \setcounter{section}{0}

\section{Introduction}\label{sec:Overview}

 Given some resource identified with the unit segment $I=[0,1]$ and a set of agents (players), one of the goals of
\emph{welfare economics} is to divide the resource among the agents in an \emph{envy-free} manner.
Envy-freeness is the principle where every player feels that their share is at least as good as the share of any other agent, and thus no player feels envy.

In the literature $I$ is  commonly referred to as the \textit{cake}.
A \textit{cut} of the cake  is a sequence of numbers $x = (x_1,\dots, x_{r-1})$ where
\begin{equation}\label{eqn:cut}
    0 \leqslant x_1 \leqslant x_2 \leqslant \dots \leqslant x_{r-1} \leqslant 1
\end{equation}
so the set of all cuts is naturally identified as the standard $(r-1)$-dimensional simplex $\Delta^{r-1}$.

The pieces of the cake arising from the \textcolor{red}{cut} (\ref{eqn:cut}) are the closed intervals (tiles)  $I_i = I_i(x) := [x_{i-1}, x_i]\,
(i=1,\dots, r)$, where $x_0 = 0$ and $x_r=1$.

\medskip
The non-degeneracy of tiles is a salient feature of some classical results such as Gale's \emph{equilibrium theorem} \cite{G}. However, as noticed by several authors in more recent publications \cite{AK, AK1, MeZe, S-H, PaZi},  this condition may be too restrictive in some situations.

 After a cut $x$ is chosen, each of the players expresses her individual \emph{preferences}, over the tiles arising from this particular cut, by pointing to one (or more) intervals which they like more than the rest.

  This idea is formalized in the following definition, where the individual preferences are interpreted as subsets of the simplex $\Delta^{r-1}$.

\begin{dfn}{\rm (\cite[Definition 2.2]{PaZi})}
The preferences of $r$ players is a matrix of subsets $(A^j_i)_{i,j=1}^r$ of the standard simplex $\Delta^{r-1}$. The subsets are interpreted as preferences in the usual sense as follows:
\begin{equation}\label{eqn:prefs-1}
x\in A^j_i \ \ \Leftrightarrow  \, \mbox{ {\rm in the cut} } x \mbox{ {\rm the player} } j \mbox{ {\rm prefers the tile} } i \, .
\end{equation}
\end{dfn}

The nature of the preferences may be unknown or hidden (black box preferences). There are however some conditions which are natural (intrinsic) or even unavoidable in the sense that they must be satisfied by all preferences.

\medskip
The necessary assumption are the following:

\begin{itemize}
\item[$\mathbf{(P_{cl})}$] The preferences (the sets $A_i^j$) are \emph{closed}.  That is, if a sequence of cuts $(x^{(n)})_{n\in \mathbb{N}}$ converges $x^{(n)}\longrightarrow x$ and if in this sequence a player always prefers the tile $I_i(x^{(n)})$, then the player also prefers the limit tile  $I_i(x)$.
\item[$\mathbf{(P_{cov})}$] The preferences form a covering: $\bigcup_{i=1}^r A_i^j=\Delta^{r-1} \ \ \forall j=1,...,r$. This means that whatever a cut is, each player  is expected to prefer at least one of the offered pieces.
\end{itemize}

If we allow players to choose degenerate tiles, the following condition is also quite natural.

\begin{itemize}
\item[$(\mathbf{P_{pe}})$] (Partition equivalence)  If two cuts produce one
and the same collection of non-degenerate tiles, the preferences in these
two cuts should be essentially the same (one can recovered from the other and vice versa). 
\end{itemize}

Expressed more explicitly this condition says the following.

\medskip
(1) A player prefers a non-degenerate tile for the first cut iff she prefers
the same tile for the second cut.

(2) A player prefers a degenerate tile for the first cut iff she prefers any
of the degenerate tiles for the second cut, and vice versa.

This condition implies, in particular, that the players do not distinguish the degenerate tiles.

\bigskip

Here is the summary of central results of the paper, presented in an informal and non-technical manner.

\medskip

\textbf{1.} (Theorem \ref{thm:prva}) Assume that $r$ is a prime power and that the preferences of $r$ players are closed $\mathbf{(P_{cl})}$ and covering $\mathbf{(P_{cov})}$. Degenerate tiles may be preferred  and the   partition equivalence condition $(\mathbf{P_{pe}})$ holds.
 One of the $r$ players is a dragon whose preferences are secret.  It is known that, once the cake is cut into $r$ tiles, the dragon will be the one who chooses the first. Then it is always possible to cut the cake in advance into $r$ tiles such that, whatever tile is taken by the dragon, the rest of the tiles may be allocated to the players so that they don't feel envy, neither to the dragon, nor to each other.  (It is self understood that the dragon, being the first to choose, feels no envy as well.)

 \medskip
 Moreover, each of the $r$ pieces can be assigned in advance at most two names of the players so that, whatever is the piece taken by the dragon, each of the players will be given a piece with her name on it. Equivalently, each human player is given in advance information about two tiles, one of which would be given to them under any circumstances.

\bigskip

\textbf{ 2. }(Theorem \ref{thm:druga})
  Let $r$ be a power of a prime. Assume that $r+1$ players have  closed $\mathbf{(P_{cl})}$, covering $\mathbf{(P_{cov})}$  preferences, which satisfy the   partition equivalence condition $(\mathbf{P_{pe}})$, and the players are allowed to choose degenerate tiles. It is known that once the cake is cut into $r$ tiles, the dragon will appear and swallow one of the players.
Then it is always possible to cut the cake in advance into $r$ tiles such that regardless of which player is taken by the dragon, the rest of the tiles may be allocated to the players in an envy-free way.

\medskip
 As in the previous theorem, each of the $r$ pieces can be assigned in advance two names of the $r+1$ players so that, whoever is the player swallowed  by the dragon, each of the surviving players will be given a piece with her name on it. Alternatively, each human player is informed in advance about at most two tiles, which include the tile that will be given to them if they are spared by the dragon.

\medskip

If degenerate tiles are forbidden,  the condition $(\mathbf{P_{pe}})$ is not needed and both Theorems 5.2 and 6.2 hold true with no assumption on $r$ whatsoever (see Theorem \ref{thm:dragon-tree} for illustration).

\subsection{Brief historical overview and predecessors of our results }

All our theorems derive inspiration from the results, both new and classical, about the \emph{envy-free-cake-cutting}.

\begin{enumerate}
\item[(1)] The original envy-free cake-cutting theorem was independently discovered by Stromquist \cite{Strom} and Woodall \cite{Wood} in 1980. It addresses the case of \emph{hungry players}, where the degenerate tiles  are never chosen by the players, and claims that (under the conditions $\mathbf{(P_{cl})}$ and $\mathbf{(P_{cov})}$) an envy-free division is always possible for any number of players $r$.
 \item[(2)] D. Gale \cite{G} is the author of the proof scheme where the result of Stromquist and Woodall is derived with the aid of the Birkhoff - von Neumann theorem about the polytope of bistochastic matrices.
 \end{enumerate}

  Much more recent are the results where the existence of envy-free divisions is guaranteed under much less restrictive conditions that the players are not necessarily hungry (which allow the players to choose degenerate tiles). However, in these results the number of players is necessarily a prime power, $r = p^k$.

  \begin{enumerate}
  \item[(3)] As shown by Avvakumov and Karasev \cite[Theorem 3.3]{AK}, see also \cite{AK1} and \cite{PaZi}, if $r$ is a prime power then an envy-free division (where all non-degenerate tiles are allocated) always exists if the preferences satisfy the conditions $\mathbf{(P_{cl})}, \mathbf{(P_{cov})}$ and $(\mathbf{P_{pe}})$.
 \item[(4)]
This result  was established earlier by Meunier and Zerbib \cite[Theorem 1]{MeZe} in the cases when $r$ is a prime or $r=4$, see also Segal-Halevi \cite{S-H} where the result was conjectured and proved in the case $r=3$.
  \item[(5)] As shown in \cite{AK} if $r$ is not a prime power the theorem of Avvakumov and Karasev is no longer true.
 \end{enumerate}

D.R. Woodall proved (already in \cite{Wood}) a refined version of the envy-free-division-theorem (for hungry players) where a \emph{secretive player} hides her preferences.
 \begin{enumerate}
 \item[(6)] Woodall's result was rediscovered by Asada et al.\ \cite{AsadaFrick17}, with a much simpler proof, under the name \emph{Strong colorful KKM theorem}.
 \item[(7)] Meunier and Su, applying the method of multilabeled Sperner lemmas, proved \cite[Corollary 1.1]{MeSu} a result where one of the players (unknown in advance) is ``voted off the group''.
 \end{enumerate}

Both (6) and (7) are proven for arbitrary $r$ with the assumption that the players are hungry (so they never choose  degenerate tiles).

\subsection{Dragon marriage lemma}

The following proposition \cite[Proposition 5.4]{postnikov_permutohedra_2009} is repeatedly used throughout the paper. It is a version of \textit{Hall's marriage theorem}  in the presence of a dragon: there are $n$ grooms and $n+1$ brides, and one of the brides is taken by a dragon. For the reader's convenience we supply a short proof.
\begin{prop}{\rm (Dragon marriage lemma)}\label{prop:Dragon-lemma}
 Let $J_1, \dots , J_{n-1} \subseteq [n]$ be not necessarily distinct sets. Then the following three conditions are equivalent:

\begin{enumerate}
\item[{\rm (1)}] For  each $k$-element subset $\{i_1,\dots, i_k\}\subseteq [n],\,  \vert J_{i_1} \cup\dots\cup J_{i_k}\vert \geq k + 1$.

\item[{\rm (2)}] For each $j\in [n]$, there is a system of distinct representatives in the family $J_1,\dots, J_{n-1}$
that avoids $j$.

\item[{\rm (3)}] There is a system of 2-element representatives $\{a_i, b_i\} \subseteq  J_i$
such that $(a_1, b_1),\dots, (a_{n-1}, b_{n-1})$ are edges of a spanning tree in $K_n$.
\end{enumerate}
\end{prop}

\begin{proof}
The statements (1) and (2) are clearly equivalent in light of  \emph{Hall's marriage theorem}. Moreover $(3) \Rightarrow (1)$ is obvious, so the interesting part of proposition is the implication $(1) \Rightarrow (3)$.

\medskip
Let $G\subseteq [n-1]\times [n]$ be a bipartite graph where
\[
     (i,j)\in G \quad \Leftrightarrow \quad j\in J_i \, .
\]
 We prove the implication $(1) \Rightarrow (3)$ by induction on the size (number of edges) of the graph $G$.

 \medskip
Let us suppose that $G$ is an inclusion minimal graph which still possesses the property (1), meaning that the \emph{Dragon marriage condition} (1) is no longer satisfied if any of the edges is removed from $G$.

\medskip
Let us call a \emph{proper subset} $I \subset [n-1]$ $(\emptyset \neq I \neq [n-1])$ \emph{tight} if $\vert G[I]\vert = \vert G\vert +1$, where $G[I]:= \{j\in [n] \mid \exists i \in I\, (i,j)\in G\}$.

\medskip
Note that the existence of tight subsets $I \subset [n-1]$ is a consequence of minimality of $G$. Indeed, for any edge $(i,j)\in E(G)$ the graph $G_1 := G\setminus \{(i,j)\}$ does not satisfy (1), therefore $\vert G_1[I]\vert \leq \vert I \vert$ and $\vert G[I]\vert = \vert I\vert +1$ for a subset $I\subset [n-1]$ which contains $i$. (The possibility that $I = [n-1]$ can be easily ruled out by choosing an edge $(i,j)$ such that $(i', j)\in G$ for some $i'\neq i$. )

\medskip
Let $I$ be a tight subset of maximal cardinality. By the inductive hypothesis there is a bipartite graph $T_1\subseteq I\times G[I]$
which  defines (in the sense of (3)) a spanning tree on $G[I]$ with edges labelled by $I$.

\medskip
Let $J := [n-1]\setminus I$ and $V:= [n]\setminus G[I]$. Note that $G[I]\cap G[J]\neq \emptyset$, as a consequence of the fact that $\vert G[I]\vert + \vert G[J]\vert \geq \vert I\vert +\vert J\vert +2 = n+1$. Choose $y\in [I]\cap G[J]$ and define $G'$ as the intersection $G' := G\cap (J\times (V\cup\{y\}))$.

\medskip
Let us show that the bipartite graph $G'$ also satisfies the \emph{Dragon marriage condition} (1). Suppose that $J_1\subset J$ is a proper subset of $J$. Then $J_1\cup I$ is a proper subset of $[n-1]$ and (since $I$ is a maximal tight subset) $\vert G[J_1\cup I]\vert \geq \vert J_1 \vert + \vert I \vert +2$. It immediately follows that $\vert G'[J_1]\vert \geq \vert G[J_1]\setminus G[I]\vert \geq \vert J_1\vert +1$.

It remains to show that $G'[J] = V\cup\{y\}$, however this is obvious since, by the previous argument, $V\subseteq G'[J_1]$ for each subset $J_1\subset J$ of cardinality $\vert J\vert -1 $.

\medskip
By the inductive hypothesis there is a spanning tree $T_2 \subset G'$ with $V\cup\{y\}$ as the set of vertices and edges labelled by $J$. The trees $T_1$ and $T_2$ have only the vertex $y$ in common, so their union $T = T_1\cup T_2$ is also a tree, with the vertex set $[n]$ and edges labelled by $[n-1]$. This completes the proof of the proposition.
\end{proof}

\section{Strong colorful KKM theorem revisited}
\label{sec:strong}

The \emph{Strong colorful KKM theorem} \cite[Theorem 3.6]{AsadaFrick17} (see also \cite[Theorem 4]{Wood} and our Section \ref{sec:strong-KKM}) offers a solution to a problem of envy-free division in the presence of a dragon (the first scenario) in the classical case of  KKM (Knaster-Kuratowski-Mazurkiewicz) preferences.

The following theorem refines that result by putting emphasis on  a \emph{decision tree}, which provides much more precise information how the pieces of the cake are distributed.

\begin{thm}\label{thm:dragon-tree}
  Let $\Delta^n = \{t\in \mathbb{R}_+^{n+1} \mid \sum_{i=1}^{n+1} = 1\}$ be an $n$-dimensional simplex with facets $\Delta^n_i = \{t\in \Delta^n \mid t_i = 0\}$.  Let $(A_i^j)$ be a $((n+1)\times n)$-matrix of closed subsets of $\Delta^n$, called the matrix of preferences. Assume that for each $j\in [n]$ the family $\{A_i^j\}_{i\in [n+1]}$ satisfies the condition of the KKM-theorem. More explicitly,
  \begin{enumerate}
    \item The family $\{A_i^j\}_{i\in [n+1]}$ is a covering of the simplex $\Delta^n$  for each $j\in [n]$.
    \item The intersection $A_i^j\cap \Delta^n_i$ is empty for all $i\in [n+1]$ and $j\in [n]$.
      \end{enumerate}
 Then one can choose for each $j\in [n]$ two distinct elements $u_j$ and $v_j$ in $[n+1]$ such that
 \begin{enumerate}
   \item[(a)] The collection $E =\{e_j\}_{j\in [n]}$ of two element sets $e_j=\{u_j, v_j\}$ is the edge-set of a tree $T=(V,E)$ (connected graph without cycles) on $V = [n+1]$ as the set of vertices.
   \item[(b)]    \begin{equation}\label{eqn:tree-intersection}
                 \bigcap_{v \mbox{ {\rm\tiny incident to} }e } A_v^e  =  \bigcap_{j\in [n]}   (A_{u_j}^j\cap A_{v_j}^j) \neq \emptyset \, .
   \end{equation}
 \end{enumerate}
\end{thm}

\proof
 Let $(O_i^j)$ be a $((n+1)\times n)$-matrix of open sets in $\Delta^n$ such that $A_i^j \subseteq O_i^j
\subseteq \Delta^n \setminus \Delta^n_i $ for each $i$ and each $j$. Let  $(f_i^j)$ be a matrix of functional preferences (fuzzy set preferences) associated to the matrix of preferences $(A_i^j)$, subordinated to $(O_i^j)$. More explicitly $(f_i^j)$ is a collection of functions $f_i^j : \Delta^n \rightarrow [0,1]$ satisfying the following conditions.
\begin{enumerate}
  \item For each $j\in [n]$ the collection $\{f_i^j\}_{i=1}^{n+1}$ is a partition of unity
  \[
          f_1^j + f_2^j + \dots + f_{n+1}^j   =1 \, .
  \]
  \item For each $i$ and $j$
  \[
           A_i^j \subseteq \{x \mid f_i^j(x) > 0\} \subseteq O_i^j \subseteq \Delta^n \setminus \Delta^n_i \, .
  \]
\end{enumerate}
Let $h_i = \frac{1}{n}(f_i^1 + f_i^2 + \dots + f_i^{n})$. The restriction of the map
\[
         h = (h_1,h_2,\dots, h_{n+1}) : \Delta^n \longrightarrow \Delta^n
\]
to the boundary $\partial \Delta^n$ is homotopic (by the linear homotopy) to the identity map. It follows that the degree of this map is equal to one and, as a consequence, $h(x) = (\frac{1}{n+1}, \dots, \frac{1}{n+1})$ for some $x\in \Delta^n$.

\medskip
Let $M$ be the matrix $M = (f_i^j(x))$ and $\Omega = (\omega_i^j) = ({\rm Sign}(f_i^j(x))$ the associated $0$-$1$ matrix of signs (${\rm Sign}(\cdot) \in \{-1, 0, +1\}$).

\medskip\noindent
{\bf Lemma.}
  The matrix $\Omega$ interpreted as a bipartite graph (with $(n+1)$ brides and $n$ grooms) satisfies the ``Dragon marriage condition'' of Postnikov \cite[Section 5]{postnikov_permutohedra_2009} (see Proposition \ref{prop:Dragon-lemma}) saying that for each subset $S\subseteq [n]$ and the corresponding set
  $\Omega[S] = \{i\in [n+1] \mid \exists j\in [n] \, f_i^j(x) > 0\}$,
  \[
                \vert \Omega[S]\vert \geq \vert S \vert + 1 \, .
  \]

Indeed,
\[
      \vert S \vert = \sum_{i\in [n+1]}^{j\in S} f_i^j  \leq \sum_{i\in \Omega[S]}^{j\in [n]} f_i^j = \frac{n}{n+1} \vert\Omega[S] \vert \quad   \Rightarrow \quad \vert\Omega[S]\vert \geq \vert S\vert +1 \, .
\]
By  Proposition \ref{prop:Dragon-lemma}   there exists a system of $2$-element representatives $e_j = \{u_j, v_j \} \subseteq \{i \mid f_i^j(x) > 0\}$, such that $E = \{e_j\}$ is the edge-set of a tree on $[n+1]$.

\medskip
The proof is completed by choosing smaller and smaller open neighborhoods $O_i^j$ of closed sets $A_i^j$ and by passing to the limit.  \qed

\subsection{Strong colorful KKM theorem}\label{sec:strong-KKM}
Theorem \ref{thm:dragon-tree} implies the standard form of the \emph{strong colorful KKM theorem} \cite[Theorem 3.6]{AsadaFrick17}  which claims that, given $n$ $KKM$-coverings $\{A_i^j\}_{j=1}^n$ of $\Delta^n$, there exists $x\in\Delta^n$ and $n+1$ bijections $\pi_i : [n] \rightarrow [n+1]\setminus \{i\}$ such that for all $i\in [n]$
\begin{equation}\label{eqn:perm-intersection}
x\in A^1_{\pi_i(1)}\cap\dots\cap A^n_{\pi_i(n)} \, .
\end{equation}

\medskip
Indeed, if a vertex $i\in [n+1]$ is chosen to be the root of the tree $T = (V,E) = ([n+1], [n])$, then there is a canonical bijection $\pi_i : V\setminus\{i\} \rightarrow E$ such that $k$ is incident to $e_k := \pi_i(k)$ for each $k\in [n+1]\setminus\{i\}$, and (\ref{eqn:perm-intersection}) is an immediate consequence of (\ref{eqn:tree-intersection}).

\bigskip
A salient feature of Theorem \ref{thm:dragon-tree} is that, given a matrix of preferences $(A_i^j)$, one should be able
to determine both $x$ and a tree $T$ (not just $x$),  which automatically leads to a very small collection of $n+1$ bijections $\pi_i$, arising from the tree $T$.

For example, just from the knowledge of $x$ and $T$, each player will know in advance which two tiles are the only candidates to be allocated to them. Moreover, from the shape of the tree they would be able to calculate the probability
of getting each of the two selected tiles, etc.

\section{Envy-free division via configuration spaces}\label{sec:prelims}

Following the general framework of \cite{PaZi},
in this section we outline a new approach to envy-free division based on equivariant topology and configuration spaces (chessboard complexes). 

\medskip
As before there are $r$ players, but now they are allowed to choose degenerate pieces.
Originally the cake $I = [0,1]$ was divided by $(r-1)$ cut-points into $r$ pieces. Now we consider cuts $x$ with a larger number of cut-points
\begin{equation}\label{eqn:cut-larger}
    0 \leqslant x_1 \leqslant x_2 \leqslant \dots \leqslant x_{2r-2} \leqslant 1
\end{equation}
together with associated allocation functions $\alpha : [2r-1] \rightarrow [r]$, which place the tiles $\{I_i\}_{i=1}^{2r-1}$
(created by the cut $x$) into $r$ ``boxes'' \cite[Section 2]{PaZi}. The pair $(x,\alpha)$ is referred to as a \emph{partition/allocation} of the cake $[0,1]$.

\medskip
We allow only \emph{``admissible''} allocation functions $\alpha : [2r-1]\rightarrow [r]$,  permitting at most one non-degenerate tile in each of the boxes. As a consequence at least $r-1$ tiles are degenerate
and many of the cut-points (\ref{eqn:cut-larger}) coincide.

\medskip

Two admissible allocation  functions $\alpha, \beta : [2r-1]\rightarrow [r]$ (and the corresponding partitions/allocations $(x,\alpha)$ and $(x,\beta)$)  are equivalent if they differ only by positions of the degenerate tiles. More explicitly
\[
    (x, \alpha) \thicksim (x,\beta)  \quad \Leftrightarrow \quad (\forall i\in [2r-1])( \, \alpha(i) \neq \beta(i) \,\Rightarrow  \, I_i \mbox{ {\rm is degenerate})}\, .
\]

\medskip
Following \cite[Section 5]{PaZi} we define $\mathcal{C}$   as the configuration space of all  equivalence classes  $[(x,\alpha)]$ of   pairs $(x,\alpha)$, where $x$ is a cut and $\alpha$ an associated admissible allocation function.

\medskip
Alternatively the configuration space  $\mathcal{C}$ can be described \cite[Proposition 5.1]{PaZi} as the chessboard complex $\Delta_{r,2r-1}$ of all non-attacking configurations of rooks in a $[r]\times [2r-1]$ chessboard. (Two rooks are in a non-attacking position if they are not allowed to be in the same row or in the same column of the chessboard.)

\medskip
In this interpretation the rows of the $[r]\times [2r-1]$-chessboard correspond to the boxes, columns correspond to the tiles, empty columns correspond to degenerate tiles. The condition ``a non-attacking collection of rooks'' is translated as ``at most one non-degenerate tile per box is allowed'', etc.

\medskip
More information about the history and applications of chessboard complexes in combinatorics and discrete geometry can be found in \cite{Mat, Z17, BLVZ, vz11, ZV92}.

\begin{dfn}{\rm (\cite[Definition 2.2]{PaZi})}
The $\mathcal{C}$-preferences of $r$ players is a matrix of subsets $(B^j_i)_{i,j=1}^r$ of the configuration space $\mathcal{C}\cong \Delta_{r,2r-1}$. More explicitly,
\begin{equation}\label{eqn:prefs-2}
\begin{split}
(x, \alpha) \in B^j_i \ \ \Leftrightarrow  \, &  \mbox{ {\rm in the cut} } x \mbox{ {\rm and the allocation} }
\alpha \mbox{ {\rm the player} } j\\
& \mbox{ {\rm prefers the content $\alpha^{-1}(i)$ of the box } } i \, .
\end{split}
\end{equation}
\end{dfn}

The symmetric group $S_r$  acts   on the configuration space $\mathcal{C}\cong \Delta_{r,2r-1}$ by renumbering the boxes
    $$\sigma(x,\alpha):= (x,\sigma \circ \alpha)$$
or equivalently, by permuting the rows of the chessboard $[r]\times [2r-1]$.

\medskip
 We say that $\mathcal{C}$-preferences are \emph{equivariant} if for each $\sigma\in S_r$,
\begin{equation}\label{eqn:equivariance-1}
  (x, \alpha ) \in B_{i}^j  \Leftrightarrow  \sigma(x, \alpha )\in B_{\sigma(i)}^j   \, .
\end{equation}
The condition (\ref{eqn:equivariance-1}) is quite natural. It follows from (\ref{eqn:prefs-2}) that $(x, \sigma\circ \alpha) \in B^j_{\sigma(i)}$ if and only if in the cut $x$ and the allocation $\sigma\circ \alpha$ the player $j$
prefers the content of the box $\sigma(i)$. Since $(\sigma\circ\alpha)^{-1}(\sigma(i)) = \alpha^{-1}(i)$ the condition (\ref{eqn:equivariance-1}) expresses the idea that players
make their decisions solely on the content of the boxes, not on their current labels.

\medskip
We say that the preferences $(B_i^j)$ are \emph{closed} if $B_i^j$ are closed subsets of the configuration space $\mathcal{C}\cong \Delta_{r,2r-1}$. The preferences $(B_i^j)$ are \emph{covering} if
$$
\cup_{i=1}^r B_i^j = \mathcal{C}
$$
for each $j\in [r]$, meaning that each player must choose one of the (possibly degenerate) tiles, displayed in the boxes by the corresponding allocation functions.

\section{First scenario: the dragon takes a piece of the cake}

The following result is a refinement of \cite[Theorem 5.1]{PaZi}.

\begin{thm}\label{thm:(r-1)}
Let $r$ be a prime power.
  Let
  $$
  (C_i^j)_{i\in [r]}^{j\in [r-1]},\ \  C_i^j\subseteq \mathcal{C} \cong \Delta_{r,2r-1}
  $$
be a $r\times (r-1)$-matrix of preferences which are closed, covering, and equivariant, in the sense of Section \ref{sec:prelims}. Then one can choose for each $j\in [r-1]$ two distinct elements $u_j$ and $v_j$ in $[r]$ such that
 \begin{enumerate}
   \item[(a)] The collection $E =\{e_j\}_{j\in [r-1]}$ of two element sets $e_j=\{u_j, v_j\}$ is the edge-set of a tree $T=(V,E)$  on $V = [r]$.
   \item[(b)]    \begin{equation}\label{eqn:tree-intersection-2}
                 \bigcap_{v \mbox{ {\rm\tiny incident to} }e } C_v^e  =  \bigcap_{j\in [r-1]}   (C_{u_j}^j\cap C_{v_j}^j) \neq \emptyset \, .
   \end{equation}
 \end{enumerate}
\end{thm}

\proof
The proof relies on methods from equivariant topology, combined with the ideas used in the proof of Theorem \ref{thm:dragon-tree}.

\medskip
 Let $(O_i^j)$ be a $(r\times (r-1))$-matrix of open sets in $\Delta_{r,2r-1}$ such that $C_i^j \subseteq O_i^j$ for all $i$ and  $j$. By averaging it can be assumed that $(O_i^j)$ is also equivariant.

 Let  $(f_i^j)$ be a matrix of functional preferences  associated to the matrix of preferences $(C_i^j)$, subordinated to $(O_i^j)$, which is also equivariant in the sense that
 \begin{equation}\label{eqn:equivariance-3-bis}
f_{\sigma(i)}^j(\sigma(x, \alpha)) := f_{\sigma(i)}^j(x, \sigma\circ\alpha)   = f_{i}^j(x, \alpha) \, .
\end{equation}
  In particular $(f_i^j)$ is a collection of functions $f_i^j : \Delta_{r,2r-1} \rightarrow [0,1]$ satisfying the following conditions.
\begin{enumerate}
  \item For each $j\in [r-1]$ the collection $\{f_i^j\}_{i=1}^{r}$ is a partition of unity
  \[
          f_1^j + f_2^j + \dots + f_{r}^j   =1  \, .
  \]
  \item For each $i$ and $j$
  \[
           C_i^j \subseteq \{x \mid f_i^j(x) > 0\} \subseteq O_i^j \, .
  \]
\end{enumerate}
Let $h_i = \frac{1}{r-1}(f_i^1 + f_i^2 + \dots + f_i^{r-1})$. The  map
\[
         h = (h_1,h_2,\dots, h_{r}) : \Delta_{r,2r-1} \longrightarrow \Delta^{r-1}
\]
is equivariant and, since $\Delta_{r,2r-1}$ is $(r-2)$-connected, Volovikov's theorem \cite{Vol96-1}
guarantees that $h(x, \alpha) = (\frac{1}{r}, \dots, \frac{1}{r})$ for some $(x,\alpha)\in \Delta_{r,2r-1}$.

\medskip
Let $M$ be the matrix $M = (f_i^j(x))$ and $\Omega = (\omega_i^j) = ({\rm Sign}(f_i^j(x))$ the associated $0$-$1$ matrix of signs.

\medskip
As in the proof of Theorem \ref{thm:dragon-tree} one obtains the inequality
 \[
                \vert \Omega[S]\vert \geq \vert S \vert + 1
  \]
for each subset $S\subseteq [r-1]$ and the corresponding set
  $$\Omega[S] = \{i\in [r] \mid \exists j\in [r-1] \, f_i^j(x) > 0\} \, .$$

As before by \cite[Proposition 5.4]{postnikov_permutohedra_2009} there exists a system of $2$-element representatives $e_j = \{u_j, v_j \} \subseteq \{i \mid f_i^j(x) > 0\}$, such that $E = \{e_j\}$ is the edge-set of a tree on $[r]$.

\medskip
Finally, the proof is completed (as in the case of Theorem \ref{thm:dragon-tree}) by choosing smaller and smaller open neighborhoods $O_i^j$ of closed sets $A_i^j$ and by passing to the limit.  \qed

\subsection{Classical preferences with a secretive player}

Theorem \ref{thm:(r-1)} allows us to prove a relative of both Theorem \ref{thm:dragon-tree} and \cite[Theorem 4.1]{AK}), where the players may choose degenerate pieces of the cake and there is a secretive or non-cooperative player (the dragon).

\begin{thm}\label{thm:prva}
Assume that  $r$  is a prime power. Assume we have the old-style closed, covering  preferences  $(A_i^j)_{i\in[r]}^{j\in [r-1]}, \ \ A_i^j \subseteq \Delta^{r-1}$ satisfying  the $(\mathbf{P_{pe}})$ condition.

 Then one can choose for each $j\in [r-1]$ two distinct elements $u_j$ and $v_j$ in $[r]$ such that
 \begin{enumerate}
   \item The collection $E =\{e_j\}_{j\in [r-1]}$ of two element sets $e_j=\{u_j, v_j\}$ is the edge-set of a tree $T=(V,E)$  on $V = [r]$ as the set of vertices.
   \item    \begin{equation}\label{eqn:tree-intersection-3}
                 \bigcap_{v \mbox{ {\rm\tiny incident to} }e } A_v^e  =  \bigcap_{j\in [r-1]}   (A_{u_j}^j\cap A_{v_j}^j) \neq \emptyset \, .
   \end{equation}
   \end{enumerate}
\end{thm}

\proof
Fundamental configuration space used in Theorem \ref{thm:(r-1)} is the chessboard complex $\Delta_{r,2r-1}$ while the corresponding object in Theorem \ref{thm:dragon-tree} is the simplex $\Delta^{r-1}$. For the comparison of these spaces and transfer from one to
another we use the diagram (\ref{eqn:transfer}), where $$I(r) := \{A \subset [0,1] \mid \{0,1\} \subseteq A \mbox{ {and} } \vert A\vert\leq r+1\}$$ is a configuration space of finite sets (symmetric product) with the topology induced by the Hausdorff metric, while $\phi$ and $\psi$ are obvious forgetful maps.

\begin{equation}
\begin{CD}\label{eqn:transfer}
 @.  \Delta^{r-1}\\
@. @V\psi VV\\
\Delta_{r,2r-1} @>\phi>> I(r)
\end{CD}
\end{equation}

Starting with the old-style (classical)  preferences $(A_i^j)$, we construct  new-style (equivariant)  preferences $(C_i^j)$ defined on the configuration space $\mathcal{C}_2 = \Delta_{r,2r-1}$.

\medskip

As in \cite[Section 4]{PaZi}, this construction can be presented in the form of an algorithm (where the diagram (\ref{eqn:transfer}) and the maps $\phi$ and $\psi$ are implicitly used but not explicitly mentioned).

\begin{enumerate}
\item Given  a cut $x$ of the segment $I$ with $2r-2$ cut points create an induced cut $y$ of $I$ with $r-1$ cut points  ($y\in \Delta^{r-1}$), preserving the same collection of non-degenerate intervals $\{I^\nu_k\}_{k=1}^s$. In other words we eliminate $r-1$ superfluous (multiple) cuts. Note that this step can be performed in many different ways.

\item If the preferences $\{A_i^j\}$ of a player dictate the choice of some non-degenerate tile $i$,   add the box $\alpha(i)$ to the preferences of the player.
 Note that the property $(\mathbf{P_{pe}})$ of $(A_i^j)$) implies that no matter how the superfluous cuts are eliminated, the result will be one and the same.
 \item If the preferences $(A_i^j)$ dictate a player to choose a degenerate interval (which occurred  after $r-1$ cuts), observe
 that there necessarily exists an empty box, since in this case the number of non-degenerate tiles is at most $r-1$. Add all empty boxes to the preferences of the player.
\end{enumerate}

\medskip
As before we emphasize  that the condition $(\mathbf{P_{pe}})$ is used in the proof that the preferences $C_i^j$ are well-defined.

\medskip
It is not difficult to check that the preferences $(C_i^j)$ satisfy all conditions of Theorem \ref{thm:(r-1)}. Then the existence of an element (partition/allocation) $(x,\alpha)$ in the intersection (\ref{eqn:tree-intersection-2}) and the corresponding tree $T$ are used, by a transfer in the opposite direction, for the proof that the intersection (\ref{eqn:tree-intersection-3}) is non-empty and the completion of the proof of Theorem \ref{thm:prva}.
 \qed

\begin{cor}
  Under the conditions of Theorem \ref{thm:prva}, there always exists a partition of $[0,1]$ into at most $r$ non-degenerate intervals such that, even if one of the intervals is removed (taken by the dragon)  each of the remaining non-degenerate intervals is given to a different player, the rest of the players are not given anything (they are given ``empty pieces''), and this distribution is envy-free from the viewpoint of each of the players (the dragon included).
\end{cor}

\section{Second scenario: the dragon takes a player }\label{sec:(r+1)}

\begin{thm}\label{thm:(r+1)}
Let $r$ be a prime power.
  Let
  $$
  (C_i^j)_{i\in [r]}^{j\in [r+1]},\ \  C_i^j\subseteq \mathcal{C} \cong \Delta_{r,2r-1}
  $$
be a $r\times (r+1)$-matrix of preferences which are closed, covering, and equivariant.

 Then one can choose for each $i\in [r]$ two distinct elements $u_i$ and $v_i$ in $[r+1]$ such that
 \begin{enumerate}
   \item The collection $E =\{e_i\}_{i\in [r]}$ of two element sets $e_i=\{u_i, v_i\}$ is the edge-set of a tree $T=(V,E)$  on vertices $V = [r+1]$.
   \item    \begin{equation}\label{eqn:tree-intersection-333}
                 \bigcap_{v \mbox{ {\rm\tiny incident to} }e } C_e^v  =  \bigcap_{i\in [r]}   (C^{u_i}_i\cap C^{v_i}_i) \neq \emptyset \, .
   \end{equation}
 \end{enumerate}
\end{thm}

\proof
The proof uses similar ideas as the proof of Theorem \ref{thm:(r+1)}.

\medskip

  We choose $(O_i^j)$,  a $(r\times (r+1))$-matrix of open sets in $\Delta_{r,2r-1}$, such that $C_i^j \subseteq O_i^j$ for all $i$ and  $j$. As before, by averaging it can be assumed that $(O_i^j)$ is also equivariant
  \[
        (x,\alpha)\in O_i^j \quad  \Leftrightarrow    \quad (x,\sigma\circ\alpha)\in O_{\sigma(i)}^j \, .
  \]

 Let  $(f_i^j)$ be a matrix of functional preferences  (associated to the matrix of preferences $(C_i^j)$), which for each $j$ forms a partition of unity subordinated to the cover $(O_i^j))_{i=1}^r$. Moreover, $(f_i^j)$ is equivariant in the sense that
 \begin{equation}\label{eqn:equivariance-3-novo}
f_{\sigma(i)}^j(\sigma(x, \alpha)) = f_{i}^j(x, \alpha) \, .
\end{equation}
  More explicitly, $(f_i^j)$ is a collection of functions $f_i^j : \Delta_{r,2r-1} \rightarrow [0,1]$ which in addition to (\ref{eqn:equivariance-3-novo}) satisfy the following conditions.
\begin{enumerate}
  \item For each $j\in [r-1]$ the collection $\{f_i^j\}_{i=1}^{r}$ is a partition of unity
  \[
          f_1^j + f_2^j + \dots + f_{r}^j   =1  \, .
  \]
  \item For each $i$ and $j$
  \[
           C_i^j \subseteq \{x \mid f_i^j(x) > 0\} \subseteq O_i^j \, .
  \]
\end{enumerate}
Let $h_i = \frac{1}{r+1}(f_i^1 + f_i^2 + \dots + f_i^{r+1})$. The  map
\[
         h = (h_1,h_2,\dots, h_{r}) : \Delta_{r,2r-1} \longrightarrow \Delta^{r-1}
\]
is equivariant and the  Volovikov's theorem again guarantees the existence of an element $(x,\alpha)\in \Delta_{r,2r-1}$ such that $h(x, \alpha) = (\frac{1}{r}, \dots, \frac{1}{r})$.

\medskip
Let $M$ be the matrix $M = (f_i^j(x))$ and $\Omega = (\omega_i^j) = ({\rm Sign}(f_i^j(x))$ the associated $0$-$1$ matrix of signs.
Note that this is a $r\times (r+1)$ matrix, as opposed to the $r\times (r-1)$ matrix arising in the proof of Theorem \ref{thm:dragon-tree}.

(Informally speaking the brides and grooms in the corresponding bipartite graph interchange the places.)

\medskip
By mimicking the proof of the lemma, used in the proof of Theorem \ref{thm:dragon-tree}, one obtains the inequality
 \[
                \vert \Omega[S]\vert \geq \vert S \vert + 1
  \]
for each subset $S\subseteq [r]$ and the corresponding set
  $$\Omega[S] = \{j\in [r+1] \mid \exists i\in [r] \, f_i^j(x) > 0\} \, .$$

As before by \cite[Proposition 5.4]{postnikov_permutohedra_2009} there exists a system of $2$-element representatives $e_i = \{u_i, v_i \} \subseteq \{j \mid f_i^j(x) > 0\}$, such that $E = \{e_i\}$ is the edge-set of a tree on $[r+1]$.

\medskip
Finally, the proof is completed (as in the case of Theorem \ref{thm:dragon-tree}) by choosing smaller and smaller open neighborhoods $O_i^j$ of closed sets $A_i^j$ and by passing to the limit.  \qed

\subsection{Classical preferences with $r+1$ players and a dragon}

Theorem \ref{thm:(r+1)} implies the following theorem which extends and refines \cite[Corollary 1.1]{MeSu}, by allowing degenerate tiles as preferences, and by specifying in advance  a decision tree describing what pieces can be  given to each of the players.

\begin{thm}\label{thm:druga}
Assume that  $r$  is a prime power. Assume we have the old-style closed, covering  preferences  $(A_i^j)_{i\in [r]}^{j\in [r+1]}, \ \ A_i^j \subseteq \Delta^{r-1}$ satisfying  the $(\mathbf{P_{pe}})$ condition.

 Then one can choose for each $i\in [r]$ two distinct elements $u_i$ and $v_i$ in $[r+1]$ such that
 \begin{enumerate}
   \item The collection $E =\{e_i\}_{i\in [r]}$ of two-element sets $e_i=\{u_i, v_i\}$ is the edge-set of a tree $T=(V,E)$  on $V = [r+1]$ as the set of vertices.
   \item    \begin{equation}\label{eqn:tree-intersection-4}
                 \bigcap_{v \mbox{ {\rm\tiny incident to} }e } A^v_e  =  \bigcap_{i\in [r]}   (A^{u_i}_i\cap A^{v_i}_i) \neq \emptyset \, .
   \end{equation}
   \end{enumerate}
\end{thm}

\proof  The proof is similar to the proof of Theorem \ref{thm:prva}, with Theorem \ref{thm:(r+1)} playing the role of Theorem \ref{thm:(r-1)}, so we omit the details. \qed

\subsection*{Acknowledgements}  Section 6  is supported by the Russian Science Foundation under grant  21-11-00040. R. \v Zivaljevi\' c was supported by the Science Fund of the Republic of Serbia, Grant No.\ 7744592, Integrability and Extremal Problems in Mechanics, Geometry and
Combinatorics - MEGIC.

\end{document}